\documentclass[12pt]{article}
\usepackage{amssymb,amsmath}

\newcommand{\qed}{\hfill \rule{2.5mm}{2.5mm}}
\newcommand{\R}{{\mathbb R}}
\newcommand{\C}{{\mathbb C}}
\newcommand{\tvect}[2]{\ensuremath{\negthinspace\begin{pmatrix}#1 \\ #2 \end{pmatrix}}}
\newcommand{\N}{{\mathbb N}}
\newcommand{\Z}{{\mathbb Z}}
\newcommand{\sgn}{{\rm sgn}\,}
\newcommand{\trace}{{\rm trace}\,}

\renewcommand{\tvect}[2]{\left(\begin{array}{cc}{ #1}\\ { #2}\end{array}\right)}
\newcommand{\proof}{{\em Proof:\ }}

\begin{document}
\newtheorem{thm}{Theorem}[section]
\newtheorem{defs}[thm]{Definition}
\newtheorem{lem}[thm]{Lemma}
\newtheorem{cor}[thm]{Corollary}
\newtheorem{prop}[thm]{Proposition}
\renewcommand{\theequation}{\arabic{section}.\arabic{equation}}
\newcommand{\newsection}[1]{\setcounter{equation}{0} \section{#1}}
%%%%%%%%%%%%% title %%%%%%%%%%%%%%%%%%%%%%%%%%%%%%%
\title{Canonical systems in $\R^2$ with periodic potentials 
     and vanishing instability intervals
      }
%%%%%%%%%%%%%%%%%%%%%%%%%%%%%%%%%%%%%%%%%%%%%%
\author{
Sonja Currie \footnote{Supported in part by the Centre for Applicable Analysis and
Number Theory and by NRF grant number IFR2011040100017.},
Thomas T. Roth, 
Bruce A. Watson \footnote{Supported in part by the Centre for Applicable Analysis and
Number Theory and by NRF grant number IFR2011032400120.} \\
School of Mathematics\\
University of the Witwatersrand\\
Private Bag 3, P O WITS 2050, South Africa }
\maketitle
%%%%%%%%%%%%%%%%%%%% abstract %%%%%%%%%%%%%%%%%%%%%%
\abstract{\noindent
 Canonical systems in $\R^2$ with absolutely continuous real symmetric
 $\pi$-periodic potentials matrices are considered.
 A through analysis of the discriminant is given along with the indexing and interlacing of the
 eigenvalues of the periodic, anti-periodic and Dirichlet-type boundary value problems on $[0,\pi]$.
 The periodic and anti-periodic eigenvalues are characterized in terms of Dirichlet type eigenvalues.
 It is shown that all instability intervals vanish if and only if the potential is 
 the product of an absolutely continuous real valued function with the identity matrix.
}
%%%%%%%%%%%%%%%%%%%%%%%%%%%%%%%%%%%%%%%%%%%%%%%
\parindent=0in
\parskip=.2in
%%%%%%%%%%%%%%% introduction %%%%%%%%%%%%%%%%%%%%%%%%%%
\newsection{Introduction\label{sec-intro}}

Borg showed in \cite{BG},  that  
a single spectrum is not sufficient to uniquely 
 determine the potential of a Sturm-Liouville problem. 
However he also showed that the spectra of 
two Sturm-Liouville problems with the same potential, $q$, 
but with one of the boundary conditions changed are sufficient to determine $q$ uniquely.
It should be noted that he allowed non-separated boundary conditions and considered Hill equations. 
Following the work of Borg, the study of inverse spectral problems developed rapidly, see
 \cite{BML} and \cite{JRM} for surveys.

Hochstadt, \cite{HH1}, considered Sturm-Liouville equations on a finite interval
with periodic or anti-periodic boundary conditions. 
He showed that if each eigenvalue 
was of multiplicity $2$ then the potential was uniquely determined as the zero potential. 
To prove this, the Sturm-Liouville problem was extended by periodicity
and the related Hill's equation studied. 
Here the analytic structure of the discriminant played a central role. 
An up to date survey of this area as well as of periodic 1-Dimensional Dirac problems
can be found in Brown, Eastham and Schmidt, \cite[pages 1-29]{BES}.
Classical results on the Hill's equation can be found in Magnus and Winkler, \cite{MW},
and on the 1-dimensional Dirac equation in Levitan and Sargsjan, \cite{LaS}.

The 1-dimensional Dirac equation arises from separation of variables in
relativistic quantum mechanics while the more general 2-dimensional canonical
system arises in classical mechanics, see for example \cite{livsic}.
The development of the theory of the 1-dimensional Dirac equation and 2-dimensional 
canonical system occurred slower, see Sargasjan and Levitan \cite{LaS}, than that of the
Sturm-Liouville equation.  For example Ambarzumyan-type theorems for Dirac operators
appeared from 1987 through 2012, \cite{MH, BML, KM, MyPu, Pu, BAW, CLW, YY}.
 Despite the parallels between Sturm-Liouville equations and canonical systems,
there are important differences: 
\begin{description}
 \item[(i)] The operators associated with canonical systems in $\R^2$ are not lower-semi-bounded,
 thus the simple variational arguments used in Sturm-Liouville theory, 
 cannot be applied directly.
\item[(ii)] canonical systems which are equivalent through a unitary transformation are
 spectrally indistinguishable, which complicates uniqueness for inverse problems. 
\item[(iii)] Oscillation theory for canonical systems is significantly more complicated
than the Sturm theory for Sturm-Liouville equations, see \cite{GL}, \cite[pages 201-207]{LaS},
\cite{JW}, \cite{Yak}.
\end{description}
In spite of (iii), intersections of  solutions $Y(z)$ with a 1-dimensional
subspace of $\R^2$ can be compared, see Teschl \cite{GT2}, and the interlacing of eigenvalues established.  We provide, for the reader's
convenience, the specific oscillation and interlacing results needed for the inverse problem.

The main theorem of this paper is Theorem \ref{main-thm} in which we consider a canonical system
 in $\R^2$ with real symmetric absolutely continuous $\pi$-periodic matrix potential.
We prove that if all instability intervals are empty, then the matrix potential is diagonal
with the two diagonal entries equal, analogous results for Hill's equation can be found in 
\cite[pages 94-111]{BES} and \cite{HH1}.

In Section 2 we give some preliminary results  on translation of the potential and the
consequential changes in the solutions to (\ref{MainDifferEqn}).  
The characteristic determinant and its properties are studied in Section 3.
The eigenvalues of the periodic and anti-periodic problems are characterized in terms of the
eigenvalues of shifted version of the Dirichlet problem (where possible) in Section 4.
The necessary asymptotic estimates are developed in Section 5. 
Finally real symmetric matrix potentials, $Q$ with absolutely continuous $\pi$-periodic
 entries, for which all instability intervals of (\ref{MainDifferEqn}) vanish,
are characterized, in Section 6, as being of the form
 $Q=qI$ where $q$ is a real (scalar) valued $\pi$-periodic absolutely continuous function.

%%%%%%%%%%% Preliminaries %%%%%%%%%%%%%%%%%%%%%%%%%%%%%
\newsection{Preliminaries\label{sec-prelim}}

Consider
\begin{eqnarray}
	\ell Y = JY' + QY,
\end{eqnarray}
where 
\begin{eqnarray*}
	J = \left( \begin{array}{cc}
0 &  1\\
-1 & 0 \end{array} \right) \quad \mbox{ and } \quad Q = \left( \begin{array}{cc}q_1 &  q\\q & q_2 \end{array} \right)
\end{eqnarray*}
in which the components of $Q(z)$ are real valued $\pi$-periodic functions on $\R$, integrable on $(0,\pi)$.
We are interested in the eigenvalue problem
\begin{eqnarray}\label{MainDifferEqn}
	\ell Y = \lambda Y 
\end{eqnarray} 
 on $[0,\pi]$, with the periodic and anti-periodic 
boundary conditions, respectively
\begin{eqnarray*}
		Y(0) &= Y(\pi),\qquad   &(BC_1) \label{per}\\  % BC_1
	Y(0) &= -Y(\pi).\qquad   &(BC_2) \label{antper} % BC_2
\end{eqnarray*}
Denote by $\mathbb{Y}$ the matrix solution of (\ref{MainDifferEqn}) 
obeying the initial condition
\begin{eqnarray}\label{initialConds}
	 [ Y_1(0) \quad Y_2(0) ] = I,
\end{eqnarray}
where $I$ is the identity matrix, and write $\mathbb{Y}=[Y_1 \quad Y_2]$.

The above boundary value problems can also be posed in the Hilbert space
$\mathbb{H} = \mathcal{L}_2(0,\pi)\times\mathcal{L}_2(0,\pi)$ with inner product
\begin{eqnarray*}
\langle Y,Z\rangle = \int_0^\pi Y(t)^T \overline{Z}(t) dt \quad \mbox{ for } Y,Z \in \mathbb{H},
\end{eqnarray*}
and norm $\| Y \|^2_2 := \langle Y,Y \rangle$.
The above boundary eigenvalue problem can be represented by the operator eigenvalue problems
\begin{eqnarray} \label{eValProb}
	L_iY = \lambda Y, \qquad  \qquad i = 1,2,
\end{eqnarray}
where $L_i = \ell|_{\mathcal{D}(L_i)}$ with domain
\begin{eqnarray}
   \mathcal{D}(L_i)=\left\{ Y=\tvect{y_1}{y_2} \,:\,y_1,y_2 \in \mbox{AC}, \ell Y \in\mathbb{H},\mbox{Y obeys } (BC_i) \right\}.
\end{eqnarray}
In addition to the operators $L_1$ and $L_2$ we define $L_3$ and $L_4$ as above but 
with the boundary conditions 
\begin{eqnarray}
	y_1(0) &=& y_1(\pi)=0,\qquad  (BC_3)\label{1Dir}\\
	y_2(0) &=& y_2(\pi)=0.\qquad (BC_4) \label{2Dir}
\end{eqnarray}
As the operators $L_j, j=1,2,3,4,$ are self-adjoint, their eigenvalues are real.  Hence we will
restrict our attention to $\lambda\in\R$.
%%%%%%%%%% Preliminaries %%%%%%%%%%%%%%%%%%%%%%%%%%%%%
\newsection{The Characteristic Determinant}

We now show that there is a, possibly multivalued, function $\rho(\lambda)$ so that for each
 $\lambda$ there is a nontrivial solution $Y$ of
 (\ref{MainDifferEqn}) on $\mathbb{R}$ with
\begin{eqnarray}\label{floquetEqn}
	Y(z + \pi,\lambda) = \rho(\lambda) Y(z,\lambda),\quad\mbox{ for all }\quad z\in \R.
\end{eqnarray}
As $\mathbb{Y}(z + \pi,\lambda)$ is a solution matrix of (\ref{MainDifferEqn})
and $\mathbb{Y}(z,\lambda)$ is a fundamental matrix of (\ref{MainDifferEqn}), 
$\mathbb{Y}(z + \pi,\lambda)$ can be written as  
\begin{eqnarray}\label{floquetTransform}
	\mathbb{Y}(z+\pi,\lambda) = \mathbb{Y}(z,\lambda)A(\lambda),
\end{eqnarray}
where $A(\lambda)$ is independent of $z$. Setting $z=0$ gives 
$A(\lambda) = \mathbb{Y}(\pi,\lambda)$.
Combining this with (\ref{floquetEqn}) and (\ref{floquetTransform}) gives that $\rho(\lambda)$ 
represents the values of $\rho$ for which
\begin{eqnarray} \label{charaEqn}
	 \mathbb{Y}(z,\lambda)(\rho I - A(\lambda))\underline{c} = 0
\end{eqnarray}
for some $\underline{c}\ne 0$, i.e. the values of $\rho(\lambda)$ are the 
eigenvalues of $A(\lambda)$. Thus the values of $\rho(\lambda)$ are the roots, $\rho$,
 of the characteristic equation
\begin{eqnarray}\label{DiracCharacteristic}
	\rho^2 - \rho\Delta(\lambda) + 1 = \det(A(\lambda) - I\rho)= 0.
\end{eqnarray}
Here
\begin{eqnarray} \label{discriminant}
	\Delta(\lambda): = y_{11}(\pi,\lambda) + y_{22}(\pi,\lambda) = \trace(A(\lambda)),
\end{eqnarray}
 is called the discriminant of (\ref{MainDifferEqn}).
In terms
of the $\Delta(\lambda)$, from (\ref{DiracCharacteristic}), $\rho(\lambda)$ is given by
\begin{eqnarray}\label{characterSolution}
	\rho(\lambda) = \frac{\Delta(\lambda) \pm \sqrt{\Delta^2(\lambda) - 4}}{2}.
\end{eqnarray}

As $Q(x)=\overline{Q(x)}$, it follows that $\Delta(\lambda)$ is real for
$\lambda\in\R$.
In this case, if $|\Delta(\lambda)|>2$ then there are two linearly independent solutions of 
(\ref{MainDifferEqn}) obeying (\ref{floquetEqn}).  Here $\rho(\lambda)$ is real and 
at least one of these has
$|\rho(\lambda)|>1$, in which case the solution has exponential growth as 
$z\to\infty$, so the solutions are unstable for such $\lambda$.
If $\lambda$ is real and $|\Delta(\lambda)|\le 2$ then there are two linearly independent solutions of 
(\ref{MainDifferEqn}) obeying (\ref{floquetEqn}) both of which have
$|\rho(\lambda)|=1$, thus making all solutions bounded for $z\in\R$ giving stability of the solution 
for such $\lambda$.
The $\lambda$-intervals on the real line for which all solutions are bounded
will be called the intervals of stability while
the intervals for which at least one solutions is unbounded
will be called instability intervals. 
The stability intervals are given by $|\Delta(\lambda)|\leq 2$ while the 
instability intervals are given by $|\Delta(\lambda)|> 2$.  It follows from  
Corollary~\ref{baw-boundary} that the instability intervals are precisely the components 
of the interior of the set on which $|\Delta(\lambda)|\ge 2, \lambda\in\R$.

The following lemma shows that $\Delta(\lambda)$ is independent of replacement of $Q(z)$ by
 $Q(z+\tau)$, that is $\Delta(\lambda)$ is independent of shifts of the independent variable
 in the potential.
This lemma is critical in our study of the inverse problem.

\begin{lem}\label{invarDiscrim}
	Let $\Delta(\lambda,\tau)$ denote the discriminant of 
	\begin{eqnarray}\label{translEval}
		JU'(z) + [Q(z + \tau) - \lambda I]U(z) = 0,
	\end{eqnarray}
	for $\tau\in\R$, then  $\Delta(\lambda,\tau)$ is independent of $\tau$.
\end{lem}

\begin{proof}
Let $U_1(z,\tau) = \tvect{u_{11}(z,\tau)}{u_{12}(z,\tau)}$ and
 $U_2(z,\tau) = \tvect{u_{21}(z,\tau)}{u_{22}(z,\tau)}$ be the solutions of (\ref{translEval}) which satisfy the initial conditions
\begin{eqnarray}
	[U_1 \, U_2](0,\tau) = I.\label{UiCondit}
\end{eqnarray}
Let $\mathbb{U}=[U_1 \, U_2]$.
Since $Y_1(z + \tau)$ and $Y_2(z + \tau)$ are solutions of (\ref{translEval}) and a basis for the solution set of (\ref{translEval}), 
we may represent $U_1(z,\tau)$ and $U_2(z,\tau)$ as a linear combination of $Y_1(z + \tau)$ and $Y_2(z + \tau)$ giving
\begin{eqnarray}\label{TransZandT}
	\mathbb{U}(z,\tau) = \mathbb{Y}(z+\tau)B(\tau),
\end{eqnarray}
where $B(\tau)$ is an invertible matrix. Inverting $B(\tau)$ and setting $z = 0$ we obtain 
\begin{eqnarray}
 B^{-1}(\tau) = \mathbb{Y}(\tau).\label{B-inverse}
\end{eqnarray}
By (\ref{discriminant}), the discriminant of problem (\ref{translEval}) is 
\begin{eqnarray}
	\Delta(\lambda,\tau) = u_{11}(\pi,\tau) + u_{22}(\pi,\tau).
\end{eqnarray} 
Combining (\ref{TransZandT}) and (\ref{B-inverse}) we get
 $$\mathbb{U}(z,\tau)\mathbb{Y}(\tau) = \mathbb{Y}(z + \tau)$$
 which when differentiated with respect to $\tau$ and $z$ gives 
\begin{eqnarray*}
	\frac{\partial \mathbb{U}(z,\tau)}{\partial \tau}B^{-1}(\tau) 
              &=& \mathbb{Y}'(z + \tau) -  \mathbb{U}(z,\tau)
	\frac{\partial}{\partial \tau}B^{-1}(\tau),   \\
	\frac{\partial\mathbb{U}(z,\tau)}{\partial z}B^{-1}(\tau) &=& \mathbb{Y}'(z + \tau). 
\end{eqnarray*}
Taking the difference of the above two equations and premultiplying by $B(\tau)$ we obtain
\begin{eqnarray*}
	\frac{\partial\mathbb{U}(z,\tau)}{\partial \tau}
    + \mathbb{U}(z,\tau)\frac{\partial\mathbb{Y}(\tau)}{\partial \tau}B(\tau) =
	 \frac{\partial\mathbb{U}(z,\tau)}{\partial z}.
\end{eqnarray*}
Now the above equation with (\ref{translEval}) and (\ref{eValProb}) yields
\begin{eqnarray}\label{UDerivTau}
   \frac{\partial\mathbb{U}(z,\tau)}{\partial \tau}
   = J(Q(z+\tau) - \lambda I)\mathbb{U}(z,\tau)
    - 	\mathbb{U}(z,\tau)J(Q(\tau) - \lambda I).
\end{eqnarray}
A direct calculation shows that
\begin{eqnarray}
  \lefteqn{\trace\{J(Q(z+\tau)-\lambda I)\mathbb{U}(z,\tau)\}\nonumber}\\
   &=& q(z+\tau)(u_{11}(z,\tau)-u_{22}(z,\tau)) + u_{12}(z,\tau)(q_2(z+\tau)-\lambda)\nonumber\\
   && - u_{21}(z,\tau)(q_1(z+\tau)-\lambda),
  \label{trace-1}\\
  \lefteqn{\trace\{\mathbb{U}(z,\tau)J(Q(\tau)-\lambda I)\}\nonumber}\\
   &=& q(\tau)(u_{11}(z,\tau)-u_{22}(z,\tau)) + u_{12}(z,\tau)(q_2(\tau)-\lambda)\nonumber\\
    && - u_{21}(z,\tau)(q_1(\tau)-\lambda).
 \label{trace-2}
\end{eqnarray}
Since $Q(\tau)=Q(\pi+\tau)$, setting $z=\pi$ in (\ref{UDerivTau}), (\ref{trace-1}) and (\ref{trace-2}) gives
\begin{eqnarray*}
  \Delta_{\tau} =  \frac{\partial}{\partial \tau}\trace\mathbb{U}(\pi,\tau) = \trace\frac{d\mathbb{U}(\pi,\tau)}{d \tau} =0.
\end{eqnarray*}
Hence $\Delta$ is independent of $\tau$, and  
$\Delta(\lambda,\tau) =\Delta(\lambda,0) = \Delta(\lambda)$
 for all $\tau\in\R$ and $\lambda\in\C$. \qed
\end{proof}

\begin{lem}\label{baw-delta}
{\bf (a)}
 The $\lambda$-derivative of $\Delta$ is given by
\begin{eqnarray}\label{DerivDeltaLambda1}
	\frac{d \Delta}{d \lambda} = y_{21}(\pi)\int^\pi_0 Y^T_1 Y_1 dt + (y_{22}(\pi)-y_{11}(\pi))\int^\pi_0 Y^T_1 Y_2 dt
	- y_{12}(\pi)\int^\pi_0 Y^T_2 Y_2 dt.
\end{eqnarray}
which can also be expressed as
\begin{eqnarray}
	\frac{d \Delta}{d \lambda} 
       = y_{12}(\pi) \left\{\frac{\Delta^2 - 4}{4y^2_{12}(\pi)} \|Y_1\|^2_2 - \left\| Y_2 -
	 \frac{y_{22}(\pi) -y_{11}(\pi)}{2y_{12}(\pi)}Y_1\right\|^2_2\right\},
         & y_{12}(\pi)\ne 0,&\label{d-der-1}\\
	\frac{d \Delta}{d \lambda} 
       = y_{21}(\pi) \left\{
 \left\| Y_1+ \frac{y_{22}(\pi) -y_{11}(\pi)}{2y_{21}(\pi)}Y_2\right\|^2_2
 -\frac{\Delta^2 - 4}{4y^2_{21}(\pi)} \|Y_2\|^2_2\right\},
         & y_{21}(\pi)\ne 0.&\label{d-der-2}
\end{eqnarray}
{\bf (b)}
    If $\Delta(\lambda)=\pm 2$ and $\frac{d\Delta}{d\lambda}(\lambda)=0$ 
     then $y_{12}(\pi)=0=y_{21}(\pi)$ and $\mp \frac{d^2\Delta}{d\lambda^2}(\lambda)>0$.
 \\
{\bf (c)}
    If $|\Delta|\le 2$ then
     \begin{eqnarray}
      \frac{1}{y_{12}(\pi)}\frac{d \Delta}{d \lambda}&<&0,\quad\mbox{for}\quad y_{12}(\pi)\ne 0,
       \label{b-1}\\
      \frac{1}{y_{21}(\pi)}\frac{d \Delta}{d \lambda}&>&0, \quad\mbox{for}\quad y_{21}(\pi)\ne 0.
       \label{b-2}
     \end{eqnarray}
 \\
{\bf (d)} If $y_{12}(\pi)=0$ or $y_{21}(\pi)=0$, then $\Delta\cdot\sgn y_{11}(\pi)\ge 2$.
\end{lem}

\proof {\bf (a)}
Taking the $\lambda$-derivative of $\mathbb{Y}$ in (\ref{eValProb}) and (\ref{initialConds}),
we obtain that $\mathbb{Y}_\lambda$ obeys the non-homogeneous initial value problem
\begin{eqnarray*}
	J\mathbb{Y}'_{\lambda}
          + (Q-\lambda I ) \mathbb{Y}_\lambda=\mathbb{Y},
\end{eqnarray*}
with the initial condition $\mathbb{Y}_\lambda(0) = {\bf 0}$.
The homogeneous equation $JY'_{\lambda} +(Q-\lambda I) Y_\lambda=0$ has the fundamental matrix solution 
$\mathbb{Y}$.
 Using the method of variation of parameters, see \cite[pp. 74]{CodaLev}, we obtain 
\begin{eqnarray}\label{Y1L1}
	\frac{\partial \mathbb{Y}(x)}{\partial \lambda}
 = -\int_0^x \mathbb{Y}(x)\mathbb{Y}^{-1}(t)J\mathbb{Y}(t)\, dt.
\end{eqnarray}
Using (\ref{Y1L1}), the $\lambda$-derivative of the discriminant (\ref{discriminant}) can rewritten as
\begin{eqnarray}\label{DerivDeltaLambda12}
	\frac{d \Delta}{d \lambda} = y_{21}(\pi)\int^\pi_0 Y^T_1 Y_1 dt + (y_{22}(\pi)-y_{11}(\pi))\int^\pi_0 Y^T_1 Y_2 dt
	- y_{12}(\pi)\int^\pi_0 Y^T_2 Y_2 dt.
\end{eqnarray}
Completing the square in the (\ref{DerivDeltaLambda12}) and using that the Wronskian 
$\det \mathbb{Y}=1$ with the definition of $\Delta$ we obtain the remaining forms for
the $\lambda$-derivative of $\Delta$.

{\bf (b)}
If $\Delta=\pm 2$ and $\frac{d\Delta}{d\lambda}=0$ then as $Y_1$ and $Y_2$ are linearly independent in $L^2(0,\pi)$, 
(\ref{d-der-1}) leads to a contradiction if $y_{12}(\pi)\ne 0$ and (\ref{d-der-2}) leads to a contradiction if
$y_{21}(\pi)\ne 0$. Thus $y_{12}(\pi)=0=y_{21}(\pi)$.

As $[y_{11}y_{22}-y_{12}y_{21}](\pi)=1$, 
it now follows that $\Delta=y_{11}(\pi)+\frac{1}{y_{11}(\pi)}$.
The function $f(t)=t+(1/t)$ on $\R\backslash\{0\}$ attains the value $2$ only at $t=1$ and the value $-2$ only at $t=-1$.
Thus $y_{11}(\pi)=y_{22}(\pi)=\pm 1$ and $\mathbb{Y}(\pi)=\pm I$.

Taking the $\lambda$-derivative of $\mathbb{Y}_\lambda$ in (\ref{eValProb}) gives
\begin{eqnarray}\label{firstInitValProb1}
	J\mathbb{Y}'_{\lambda\lambda} + (Q-\lambda I)\mathbb{Y}_{\lambda\lambda} 
= 2\mathbb{Y}_{\lambda},
\end{eqnarray}
and we obtain that $\mathbb{Y}_{\lambda\lambda}$ obeys the initial condition 
$\mathbb{Y}_{\lambda\lambda}(0)=0$.
Using the method of variation of parameters as in (\ref{Y1L1}) gives
\begin{eqnarray}
	\frac{1}{2}\frac{\partial^2 \mathbb{Y}}{\partial\lambda^2}(x)
       = \int_0^x \int_0^t 
\mathbb{Y}(x)\mathbb{Y}^{-1}(t)J\mathbb{Y}(t)\mathbb{Y}^{-1}(\tau)J
\mathbb{Y}(\tau) \,d\tau\,dt,
\label{Lam2Soln1}
\end{eqnarray}
which with $x=\pi$ and $\mathbb{Y}(\pi)=\pm I$ yields
\begin{eqnarray}
	\frac{1}{2}\frac{\partial^2 \mathbb{Y}}{\partial\lambda^2}(\pi)
       = \pm \int_0^\pi 
\mathbb{Y}^{-1}(t)J\mathbb{Y}(t)
 \int_0^t 
\mathbb{Y}^{-1}(\tau)J \mathbb{Y}(\tau) \,d\tau\,dt.
\label{Lam2Soln12}
\end{eqnarray}
Here
$$\mathbb{Y}^{-1}J\mathbb{Y}=\left[\begin{array}{cc} Y_2^TY_1 & Y_2^TY_2 \\ -Y_1^TY_1 & -Y_1^TY_2 \end{array}\right]$$
giving
\begin{eqnarray*}
	\frac{\pm 1}{2}\frac{d^2 \Delta }{d\lambda^2}(\lambda)
	&=&\frac{\pm 1}{2}{\rm trace}\left(\frac{\partial^2 \mathbb{Y}}{\partial\lambda^2}(\pi)
          \right)\\
       &=&
     -\int_0^\pi Y^T_2Y_2\int_0^xY_1^TY_1\,dt\,dx + 
2\int_0^\pi Y_2^TY_1\int_0^xY_2^TY_1\,dt\,dx\\
       && -  \int_0^\pi Y^T_1Y_1\int_0^xY_2^TY_2\,dt\,dx.
\end{eqnarray*}
As $Y_1, Y_2$ have real entries for $\lambda\in\R$, by Fubini's Theorem applied to 
the above double integrals we obtain
\begin{eqnarray*}
	\frac{\pm 1}{2}\frac{d^2 \Delta }{d\lambda^2}(\lambda)
       &=&
     -\int_0^\pi Y^T_2Y_2\,dt\,\int_0^\pi Y_1^TY_1\,dt + 
\left(\int_0^\pi Y_2^TY_1\,dt\right)^2\\
	&=& -\|Y_1\|^2_2\|Y_2\|^2_2 + \langle Y_1,Y_2\rangle^2 < 0,
\end{eqnarray*}
for $\lambda\in\R$. Now H\"older's inequality gives that
$Y_1$ and $Y_2$ are linearly independent.

{\bf (c)} If $|\Delta|\le 2$ then $\Delta^2-4\le 0$ so (\ref{d-der-1}) and (\ref{d-der-2})
  respectively yield (\ref{b-1}) and (\ref{b-2}). 

{\bf (d)} If $y_{12}(\pi)=0$ or $y_{21}(\pi)=0$ then as ${\rm det}{\mathbb Y}(\pi)=1$,
  it follows that $y_{11}(\pi)y_{22}(\pi)=1$ giving $\Delta=y_{11}(\pi)+\frac{1}{y_{11}(\pi)}$ so
  $\Delta\ge 2$ if $y_{11}(\pi)>0$ and $\Delta\le -2$ if $y_{11}(\pi)<0$.
\qed

\begin{cor}\label{baw-boundary}
 For $\lambda\in\R$, the function $|\Delta(\lambda)|$ attains the value $2$ only on the boundary 
of the set $\Gamma=\{\lambda\in\R\,|\,|\Delta(\lambda)|\ge 2\}$.  
\end{cor}

\proof
 Suppose that $\lambda\in{\rm int}(\Gamma)$ and $\Delta(\lambda) = \pm 2$.
 As $\lambda\in{\rm int}(\Gamma)$ there is $\delta>0$ so that
  $I:=(\lambda-\delta,\lambda+\delta)\subset \Gamma$.
  The continuity of $\Delta$ and  connectedness of $I$ give that 
  $\pm \Delta\ge 2$ on $I$. Hence $\pm\Delta$ attains a local minimum at $\lambda$.
  Thus $\Delta'(\lambda)=0$. 
  Lemma \ref{baw-delta}(b) can now be applied to give
  $\pm\Delta''(\lambda)<0$. From the analyticity of $\Delta$,  $\Delta''$ is continuous, making
  $\pm\Delta''<0$ on a neighbourhood, say $N$, of $\lambda$.
  Hence $\pm\Delta<2$ on $N\backslash\{\lambda\}$, which contradicts $\pm\Delta\ge 2$ on $I$.
\qed
%%%%%%%%%%% Characterisation %%%%%%%%%%%%%%%%%%%%%%%%%%
\newsection{Eigenvalues}

Let $\Psi(z) = \left(\begin{array}{c}\psi_1(z) \\ \psi_2(z)\end{array}\right)$ be the
 non-trivial solution
 of (\ref{MainDifferEqn}) satisfying the initial condition
$\left( \begin{array}{c} \psi_1(0)\\ \psi_2(0) \end{array} \right) = 
 \left( \begin{array}{c} \cos\gamma\\ \sin\gamma \end{array} \right)$ where $\gamma \in [0,\pi)$. 
Define $R(z,\lambda,\gamma)$ and $\theta(z,\lambda,\gamma)$ by
\begin{eqnarray}
      \Psi(z)  = \left( \begin{array}{c}
      R(z,\lambda,\gamma)\cos\theta(z,\lambda,\gamma) \\
      R(z,\lambda,\gamma)\sin\theta(z,\lambda,\gamma)  \end{array} \right),
\end{eqnarray}
where $R(z,\lambda,\gamma)>0$ and $\theta(z,\lambda,\gamma)$
 is a continuous function of $z$ with
$\theta(0,\lambda,\gamma) = \gamma$.
From now on $\theta$ will be referred to as the angular part of $\Psi$.
The function $R(z,\lambda,\gamma)$ is differentiable in $z, \lambda, \gamma,$
and $\theta(z,\lambda,\gamma)$ is analytic in $\lambda$ and $\gamma$ for fixed $z$, and
differentiable in $z$ for fixed $\lambda$ and $\gamma$.
Here $\theta(z,\lambda,\gamma)$ is the solution to a first order initial value problem
\begin{eqnarray}
   \theta'&=&\lambda -q\sin 2\theta -q_{1}\cos^2\theta-q_{2}\sin^2\theta,\label{theta-eq}\\
   \theta(0)&=&\gamma.\label{theta-ic}
\end{eqnarray}
This initial value problem
obeys the conditions of \cite[Section 69.1]{mcshane}, from which it follows that 
$\theta(z,\lambda,\gamma)$ is jointly continuous in $(z,\lambda,\gamma)$.
Moreover, for fixed $z>0$ and $\gamma$, $\theta(z,\lambda,\gamma)$ is strictly increasing in 
$\lambda, \lambda\in\R$, see Weidmann \cite[p. 242]{JW}, 
with $\theta(z,\lambda,\gamma)\to\pm\infty$
as $\lambda\to\pm\infty$, see \cite{BV}.  
Thus the eigenvalues,  $\nu_n, n\in\Z,$ and $\mu_n, n\in \Z$,
of $L_3$ and $L_4$, respectively, are simple and determined uniquely by the equations
\begin{eqnarray}
  \theta(\pi,\nu_n,\pi/2)&=&n\pi+\frac{\pi}{2},\quad n\in\Z,\label{def-nu}\\
  \theta(\pi,\mu_n,0)&=&n\pi,\quad n\in\Z.\label{def-mu}
\end{eqnarray}
As a consequence of the above observation it follows that
$\mu_n, \nu_n, \to \pm\infty$ as $n\to\pm\infty$.

\begin{lem}\label{stability}
\begin{itemize}
\item[(a)] For each $n\in\Z$, 
\begin{eqnarray}
\max\{\mu_n,\nu_n\}<\min\{\mu_{n+1},\nu_{n+1}\}.\label{stability-1}
\end{eqnarray}
\item[(b)] If
$\lambda\in (\min\{\nu_n,\mu_n\},\max\{\nu_{n+1},\mu_{n+1}\})$ and
$|\Delta(\lambda)|\le 2$ then $$(-1)^n\Delta'(\lambda)<0.$$
\item[(c)]
The set $|\Delta(\lambda)|\ge 2$ consists of a countable union of disjoint closed finite intervals,
each of which contains precisely one of the sets $\{\nu_n, \mu_n\}, n\in\Z$. 
The end points of these intervals as the only points at which $|\Delta(\lambda)|= 2$. 
\end{itemize}
\end{lem}

\begin{proof}
{\bf (a)}
For fixed $\lambda$, $\theta(\pi,\lambda,\gamma)$ is monotonic increasing in $\gamma$
(this follows from the fact that $\theta$ is a solution to a first order differential equation
which has a unique solution for each initial value - giving that if a solution $\theta_1$ begins
below $\theta_2$ then it remains below $\theta_2$ for all values of the independent variable).
Thus
$$\theta(\pi,\mu_n,\pi/2)<\theta(\pi,\mu_n,\pi)=(n+1)\pi<(n+1)\pi+\frac{\pi}{2}
=\theta(\pi,\nu_{n+1},\pi/2),$$
which, as $\theta(\pi,\lambda,\pi/2)$ is increasing in $\lambda$, gives $\mu_n<\nu_{n+1}$.
As $\theta(\pi,\lambda,0)$ is increasing in $\lambda$, $\nu_n<\nu_{n+1}$.  Combining these
inequalities gives $\max\{\mu_n,\nu_n\}<\nu_{n+1}$.
Similarly 
$$\theta(\pi,\nu_n,0)<\theta(\pi,\nu_n,\pi/2)=n\pi+\frac{\pi}{2}<(n+1)\pi
=\theta(\pi,\mu_{n+1},0)$$
giving $\nu_n<\mu_{n+1}$ and $\mu_n<\mu_{n+1}$ giving 
$\max\{\mu_n,\nu_n\}<\mu_{n+1}$. Hence (\ref{stability-1}) follows.

{\bf (b)} From the monotinicity of $\theta(\pi,\lambda,\pi/2)$ in $\lambda$,
for $\lambda\in (\nu_n,\nu_{n+1})$,
$$n\pi+\frac{\pi}{2}=\theta(\pi,\nu_n,\pi/2)
   <\theta(\pi,\lambda,\pi/2)
   <\theta(\pi,\nu_{n+1},\pi/2)=(n+1)\pi+\frac{\pi}{2},$$
giving
\begin{eqnarray}
 (-1)^n y_{21}(\pi,\lambda)=(-1)^n R(\pi,\lambda,\pi/2)\cos\theta(\pi,\lambda,\pi/2)<0.
  \label{2014-sign-a}
\end{eqnarray}

Similarly, for $\lambda\in (\mu_n,\mu_{n+1})$,
$$n\pi=\theta(\pi,\mu_n,0)
   <\theta(\pi,\lambda,0)
   <\theta(\pi,\mu_{n+1},0)=(n+1)\pi,$$
giving
\begin{eqnarray}
 (-1)^n y_{12}(\pi,\lambda)=(-1)^n R(\pi,\lambda,0)\sin\theta(\pi,\lambda,0)>0.
\label{2014-sign-b}
\end{eqnarray}

From (\ref{stability-1}) we have that $(\nu_n,\nu_{n+1})\cap (\mu_n,\mu_{n+1})\ne \phi$ and
thus
$$(\nu_n,\nu_{n+1})\cup (\mu_n,\mu_{n+1})=
(\min\{\nu_n,\mu_n\},\max\{\nu_{n+1},\mu_{n+1}\}).$$
Now by Lemma \ref{baw-delta}(c) along with   (\ref{2014-sign-a}) and (\ref{2014-sign-b}),
if 
$$\lambda\in (\min\{\nu_n,\mu_n\},\max\{\nu_{n+1},\mu_{n+1}\})\quad\mbox{and}\quad
|\Delta(\lambda)|\le 2$$ then $(-1)^n\Delta'(\lambda)<0$.

{\bf (c)}
Since $|\Delta(\lambda)|$ is continuous,
the set of $\lambda\in\R$ for which
 $|\Delta(\lambda)|\ge 2$ consists of a countable union of disjoint closed finite intervals.
From the definition of $\nu_n$, we have $y_{21}(\pi,\nu_n)=0$ and 
$y_{22}(\pi,\nu_n)=(-1)^nR(\pi,\nu_n,\pi/2)$.
Hence $y_{11}(\pi,\nu_n)=(-1)^n/R(\pi,\nu_n,\pi/2)$ and
 $(-1)^n\Delta(\nu_n)\ge 2$.
Similarly $y_{11}(\pi,\mu_n)=(-1)^nR(\pi,\mu_n,0)$ and 
$y_{12}(\pi,\mu_n)=0$. Hence
$y_{22}(\pi,\mu_n)=(-1)^n/R(\pi,\mu_n,0)$ and 
 $(-1)^n\Delta(\mu_n)\ge 2$. 
Hence, for each $n\in \Z$,
\begin{eqnarray}
 \min\{(-1)^n\Delta(\min\{\nu_n,\mu_n\}), (-1)^n\Delta(\max\{\nu_n,\mu_n\})\}\ge 2. 
 \label{min-max}
\end{eqnarray}
Let
$$S:=\{\lambda|(-1)^n\Delta(\lambda)< 2\}\cap (\min\{\nu_n,\mu_n\},\max\{\nu_{n},\mu_{n}\}).$$
If $S\ne\emptyset$ then there is  $\lambda^*\in S$.
Here $K:=(-1)^n\Delta(\lambda^*)<2$ and by (\ref{min-max}),
  $(-1)^n\Delta(\max\{\nu_n,\mu_n\})\ge 2$.
So from the intermediate value theorem there is $\lambda$ with
$\lambda^*\le\lambda\le\max\{\nu_n,\mu_n\}$ having $(-1)^n\Delta(\lambda)=(2+K)/2$.
The set of such $\lambda$ is compact and thus has a least element, say $\lambda^\dagger$.
By part (b) of this lemma $(-1)^n\Delta'(\lambda)<0$
for all $\lambda^*\le \lambda\le \lambda^\dagger$ giving the contradiction
$$K=(-1)^n\Delta(\lambda^*)\ge (-1)^n\Delta(\lambda^\dagger)=(2+K)/2.$$
Thus $S=\emptyset$ and for each $n\in\Z$ both $\mu_n$ and $\nu_n$
lie in the same component of $\{\lambda | |\Delta(\lambda)|\ge 2\}$. 
Due to the sign alternation in (\ref{min-max}) as $n$ changes, 
each component of 
$\{\lambda\in\R| |\Delta(\lambda)|\ge 2\}$
contains at most one pair $\{\mu_n,\nu_n\}, n\in\Z$.

It remains to show that every component of $\{\lambda\in\R| |\Delta(\lambda)|\ge 2\}$
contains $\mu_n$ for some $n\in\Z$.
 If not then there is a component, say $T$,
of  $\{\lambda\in\R| |\Delta(\lambda)|\ge 2\}$ and $n\in\Z$ so that 
$T\subset (\mu_n,\mu_{n+1})$. 
Let $[\tilde{\lambda}_{-1},\tilde{\lambda}_0]$ and
$[\tilde{\lambda}_{3},\tilde{\lambda}_4]$ denote the components of 
 $\{\lambda\in\R| |\Delta(\lambda)|\ge 2\}$ containing
$\mu_n$ and $\mu_{n+1}$ respectively.
The set $T:=[\tilde{\lambda}_1,\tilde{\lambda}_2]$ is compact and we may,
 without loss of generality,
assume that $\tilde{\lambda}_1$ is the least $\lambda>\tilde{\lambda}_0$ 
with $|\Delta(\lambda)|\ge 2$.
Here 
$\tilde{\lambda}_0<\tilde{\lambda}_1\le \tilde{\lambda}_2<\tilde{\lambda}_{3}$. 
From (\ref{def-mu}) we have $(-1)^n\Delta(\mu_n)\ge 2$, however,
from part (b) of this lemma, $(-1)^n\Delta'(\lambda)<0$ for 
$\lambda\in (\tilde{\lambda}_0,\tilde{\lambda}_1)$. 
Thus $\Delta(\lambda)\le -2$ for $\lambda\in T$.
Again, as $(-1)^n\Delta'(\lambda)<0$ for 
$\lambda\in (\mu_n,\mu_{n+1})\backslash T$, $\Delta(\lambda)\le -2$ for all
$\lambda\in[\tilde{\lambda_1},\mu_{n+1}]$. Hence $T$ contains $\mu_{n+1}$,
contradicting the definition of $T$, and giving that no such $T$ exists.

The last part of the claim follows directly from 
Corollary \ref{baw-boundary}.
\qed
\end{proof}

 We denote the components (maximal connected subsets) of the set
$\{\lambda\in\R| \Delta(\lambda)\ge 2\}$ by $[\lambda_{2k-1},\lambda_{2k}]$ indexed so that
$\{\mu_{2k},\nu_{2k}\}\subset [\lambda_{2k-1},\lambda_{2k}]$ 
(this indexing is possible and uniquely defined by the previous lemma).
Similarly we denote the components of the set
$\Delta(\lambda)\le -2$ by $[\lambda_{2k-1}',\lambda_{2k}']$, labeled so that
$\{\mu_{2k-1},\nu_{2k-1}\}\subset [\lambda_{2k-1}',\lambda_{2k}']$.
With this indexing
\begin{eqnarray}\label{indexing}
\lambda'_{2k-1}\le\{\mu_{2k-1},\nu_{2k-1}\}
 \le \lambda'_{2k}<\lambda_{2k-1}\le
\{ \mu_{2k},\nu_{2k}\}\le \lambda_{2k}<\lambda'_{2k+1}.\label{InterlaceEvalues}
\end{eqnarray}
 Here by $\lambda'_{2k-1}\le\{\mu_{2k-1},\nu_{2k-1}\}
 \le \lambda'_{2k}$ we mean that both $\mu_{2k-1}$ and $\nu_{2k-1}$ are
 greater than or equal to $\lambda'_{2k-1}$ and less than or equal to $\lambda'_{2k}$
 with analogous interpretation for
 $\lambda_{2k-1}\le \{ \mu_{2k},\nu_{2k}\}\le \lambda_{2k}$.
 The instability intervals are thus $I_{2k}:=(\lambda_{2k-1},\lambda_{2k})$
 and $I_{2k-1}=(\lambda_{2k-1}', \lambda_{2k}')$, $k\in\Z$, which might be the an empty interval.
 From (\ref{characterSolution}) the solutions of $\Delta(\lambda)=2$ and $\Delta(\lambda)=-2$ are
 the eigenvalues of the periodic and anti-periodic problems respectively,
 as these are respectively where $\rho(\lambda) = 1$ and $\rho(\lambda) = -1$.
Hence the eigenvalues of $L_1$ and $L_2$ are $(\lambda_j)$ and $(\lambda'_j)$ respectively. 
This can be visualized as follows.

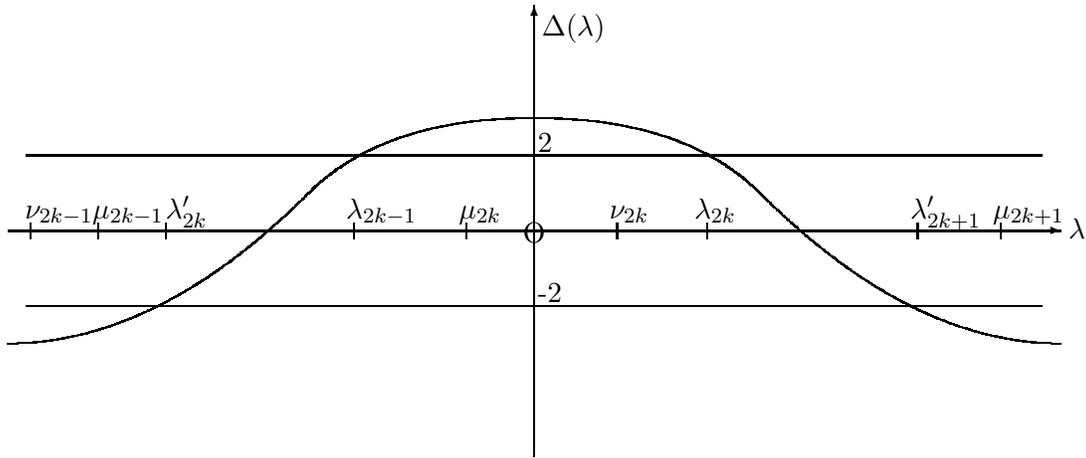
\setlength{\unitlength}{1cm}
\begin{picture}(6,6)
  \put(0.25, 2){\line(1,0){13.5}}
  \put(0.25, 4){\line(1,0){13.5}}
  \put(0, 3){\vector(1, 0){14}}
  \put(7, 0){\vector(0, 1){6}}
  \put(-1, .5){\qbezier(5,3)(6,4)(8,4)  \qbezier(5,3)(3, 1)( 1,1)}
  \put(15, .5){\qbezier(-5,3)(-6,4)(-8,4)  \qbezier(-5,3)(-3, 1)( -1,1)}
  \put(14.1,2.9){$\lambda$}
  \put(7.1,5.6){$\Delta(\lambda)$}
  \put(6.85,2.85){O}
  \put(7.05, 2.05){-2}
  \put(7.05, 4.05){2}

  \put(.3,2.9){\line(0,1){.2}}
  \put(.1,3.17){ $\nu_{2k-1}$}

  \put(1.2,2.9){\line(0,1){.2}}
  \put(1,3.17){ $\mu_{2k-1}$}

  \put(2.1,2.9){\line(0,1){.2}}
  \put(2.1,3.17){$\lambda'_{2k}$}

  \put(4.6,2.9){\line(0,1){.2}}
 \put(4.5,3.17){$\lambda_{2k-1}$}
 
  \put(6.1,2.9){\line(0,1){.2}}
  \put(6,3.17){$\mu_{2k}$}

  \put(8.1,2.9){\line(0,1){.2}}
  \put(8,3.17){$\nu_{2k}$}

  \put(9.3,2.9){\line(0,1){.2}}
  \put(9,3.17){ $\lambda_{2k}$}

  \put(12.1,2.9){\line(0,1){.2}}
  \put(12,3.17){$\lambda'_{2k+1}$}

  \put(13.2,2.9){\line(0,1){.2}}
  \put(13.1,3.17){$\mu_{2k+1}$}
 \end{picture}

\begin{cor}\label{doubleEVal}
The eigenvalue $\lambda_{2k}$ (resp. $\lambda'_{2k}$) is a double eigenvalue if and only if the
 interval $[\lambda_{2k-1},\lambda_{2k}]$
 $(\mbox{resp. }[\lambda'_{2k-1},\lambda'_{2k}])$ is reduced to a single point.
\end{cor}
\begin{proof}
	If the interval $[\lambda_{2k-1},\lambda_{2k}]$ reduces to a single point then
         $\lambda_{2k-1}=\mu_{2k} = \nu_{2k} = \lambda_{2k}$ giving
         $y_{12}(\pi)=0=y_{21}(\pi)$. Thus, ${\mathbb Y}(\pi,\lambda_{2k})$ 
         is diagonal with trace
         $2=\Delta(\lambda_{2k})=y_{11}(\pi)+y_{22}(\pi)$
         and determinant $1=[y_{11}y_{22}-y_{12}y_{21}](\pi)=y_{11}(\pi)y_{22}(\pi)$.
         Hence ${\mathbb Y}(\pi,\lambda_{2k})=I$.
         Thus $Y_1$ and $Y_2$ are both periodic eigenfunctions and the
         eigenspace attains its maximal dimension of $2$.

         Conversely if $\lambda_{2k}$ is a double eigenvalue then
         all solutions are $\pi$-periodic as the solution space is only $2$-dimensional.  In particular
         $Y_1$ and $Y_2$ are eigenfunctions. Thus $y_{11}(\pi)=1=y_{22}(\pi)$ and 
         $y_{12}(\pi)=0=y_{21}(\pi)$ giving $\Delta(\lambda_{2k})=2$.
         Now by Lemma \ref{baw-delta}(a)
         $\Delta'(\lambda_{2k})=0$ but by Lemma \ref{baw-delta}(b)
         $\Delta''(\lambda_{2k})< 0$ so the interval $[\lambda_{2k-1},\lambda_{2k}]$
         reduces to a single point.
         
         Similar reasoning can be applied to the case of $\lambda'_{2k}$.
\qed        
\end{proof}

We now turn our attention back to the translated equation (\ref{translEval}).

\begin{thm}\label{thm-bounds} 
Let $\mu_i(\tau)$ denote the eigenvalue $\mu_i$ but for the differential
  equation in which $Q(z)$ has been replaced by the shifted potential $Q(z+\tau)$.
  In terms of the above eigenvalues, for $k\in\Z$, we obtain
	\begin{eqnarray*}
		\begin{array}{rclrcl}
		\lambda_{2k-1} &=& \min_{\tau}\mu_{2k}(\tau),\, k\ne 0, \qquad
 &  \qquad \lambda_{2k} &=& \max_{\tau}\mu_{2k}(\tau),\, k\ne 0,\\
		\lambda'_{2k-1} &=& \min_{\tau}\mu_{2k-1}(\tau) \qquad
 & \qquad \lambda'_{2k} &=& \max_{\tau}\mu_{2k-1}(\tau). \end{array}
	\end{eqnarray*}
\end{thm}

\begin{proof}
     From Lemma \ref{invarDiscrim}, the eigenvalues $\lambda_i, \lambda'_i, i\in\Z,$
     are independent of $\tau$.
     Let $\Phi_\tau(z,\xi,\lambda,\gamma)$ be the solution of the equation (\ref{theta-eq})
     with initial condition 
     $\Phi_\tau(\xi,\xi,\lambda,\gamma)=\gamma$ and 
     $Q(z)$ replaced by $Q(z+\tau)$.
    Here $\Phi_\tau(z,\xi,\lambda,\gamma)$ is continuous in $(z,\xi,\lambda,\gamma)$ 
     by \cite[Section 69.1]{mcshane}.
     In addition, as
    $$\Phi_\tau(z,\xi,\lambda,\gamma)
    =\Phi_0(z+\tau,\xi+\tau,\lambda,\gamma),$$
    it follows that $\Phi_\tau$ is continuous in $\tau$, and
    \begin{eqnarray*}
      \Phi_\tau(\pi,0,\mu_n(\tau),0)&=&n\pi,\quad n\in\Z,
    \end{eqnarray*}
     defines $\mu_n(\tau)$.
   
     As for $\theta(z,\lambda,\gamma)$, the derivative of $\Phi_\tau(z,\xi,\lambda,\gamma)$
    with respect to $\lambda$ is positive. 
     Thus the inverse function theorem applied to $\Phi_\tau(z,\xi,\lambda,\gamma)$ gives
    that $\mu_n(\tau)$ is continuous in $\tau$.
     Now from Lemma \ref{invarDiscrim} the sets $\{\lambda_i | i\in\Z\}$ and
      $\{\lambda'_i | i\in\Z\}$ do not depend on $\tau$, while, from the
     continuity of $\mu_n(\tau)$, the indexing of the eigenvalues $\lambda_i, \lambda'_i$
     does not depend of $\tau$.
     Hence $\mu_{2k}(\tau)\in [\lambda_{2k-1},\lambda_{2k}]$, for all $\tau$, giving
     $$\lambda_{2k-1}\le \inf_\tau \mu_{2k}(\tau)\le \sup_\tau \mu_{2k}(\tau)\le \lambda_{2k}$$
   and
     $$\lambda'_{2k-1}\le \inf_\tau \mu_{2k-1}(\tau)\le \sup_\tau \mu_{2k-1}(\tau)\le 
      \lambda'_{2k}.$$

   If $Y$ is an eigenfunction to the periodic eigenvalue $\lambda_{2k-1}$ then $Y$
     has angular part $\theta(x,\lambda_{2k-1},\gamma)$ where
     without loss of generality $\gamma\in [0,\pi)$.
     Now $\mu_{2k-1}\le \lambda'_{2k}<\lambda_{2k-1}\le\mu_{2k}$.
   For $k\ge 1$,
    as $\theta(x,\lambda,\gamma)$ is increasing in $\gamma$,
 we have
     $$\theta(0,\lambda_{2k-1},\gamma)=\gamma<\pi\le (2k-1)\pi=\theta(\pi,\mu_{2k-1},0)
      <\theta(\pi,\lambda_{2k-1},0)\le \theta(\pi,\lambda_{2k-1},\gamma)$$
     so by the intermediate value theorem there exists $\tau\in (0,\pi]$ with 
      $\theta(\tau,\lambda_{2k-1},\gamma)=\pi$.
     As $Y$ is $\pi$-periodic, so is $Y(x+\tau)$.  Thus
     $\lambda_{2k-1}=\mu_n(\tau)$ for some $n$, but the only $n$ for which 
     $\mu_n(\tau)$ is in $[\lambda_{2k-1},\lambda_{2k}]$ is $n=2k$.
     Hence
  $\lambda_{2k-1}= \min_\tau \mu_{2k}(\tau)$.
    In the case of $k\le -1$ we have
     $$\theta(\pi,\lambda_{2k-1},\gamma)< \theta(\pi,\lambda_{2k-1},\pi)
        =\pi+ \theta(\pi,\lambda_{2k-1},0)\le \pi+ \theta(\pi,\mu_{2k},0)=(2k+1)\pi\le -\pi.$$
     But  $0\le\gamma=\theta(0,\lambda_{2k-1},\gamma)$
     so there exists $\tau\in [0,\pi)$ such that $\theta(\tau,\lambda_{2k-1},\gamma)=0$.
     Proceeding as in the previous case, $\lambda_{2k-1}=\mu_{2k}(\tau)$
     and  $\lambda_{2k-1}= \min_\tau \mu_{2k}(\tau)$.

          For $k\in\Z$, we have that $\mu_{2k}\le\lambda_{2k}<\mu_{2k+1}$.
   If $Y$ is an eigenfunction to the periodic eigenvalue $\lambda_{2k}$ then $Y$
     has angular part $\theta(x,\lambda_{2k},\gamma)$ where
     without loss of generality $\gamma\in [0,\pi)$.
     For $k\ge 1$,
     $$\theta(0,\lambda_{2k},\gamma)=\gamma<\pi< 2k\pi=\theta(\pi,\mu_{2k},0)
      \le\theta(\pi,\lambda_{2k},0)\le \theta(\pi,\lambda_{2k},\gamma)$$
     so there exists $\tau\in (0,\pi]$ for which $\theta(\tau,\lambda_{2k-1},\gamma)=\pi$
     and $\lambda_{2k}=\mu_{2k}(\tau)$.
    In the case of $k\le -1$ we have
     $$\theta(\pi,\lambda_{2k},\gamma)< \theta(\pi,\lambda_{2k},\pi)
        =\pi+ \theta(\pi,\lambda_{2k},0)< \pi+ \theta(\pi,\mu_{2k+1},0)=(2k+1)\pi\le -\pi.$$
    Now $-\pi<0\le \gamma=\theta(0,\lambda_{2k},\gamma)$
 so there exists $\tau\in [0,\pi)$ with $\theta(\tau,\lambda_{2k},\gamma)=0$
     giving $\lambda_{2k}=\mu_{2k}(\tau)$. Thus for $k\in\Z\backslash\{0\}$,
    $\lambda_{2k}=\max_\tau \mu_{2k}(\tau)$.
    
  For an eigenfunction of the $Y$ of the anti-periodic problem at eigenvalue $\lambda'_j$, where
  $j=2k-1$ or $2k$,
 we have $Y(0)=-Y(\pi)$ giving that the angular part $\theta(x,\lambda'_j,\gamma)$ of $Y$ 
  necessarily changes by an odd multiple of $\pi$ over the interval $[0,\pi]$.  
  In particular this ensures that there is some $\tau \in [0,\pi]$ for which
  $\theta(\tau,\lambda'_j,\gamma)= \pm\pi$. Setting $Z(x)=Y(x)$ for $x\in [0,\pi]$ and
  $Z(x)=-Y(x-\pi)$ for $x\in (\pi,2\pi]$ we have that $Z$ is a solution of the periodically
  extended equation on $[0,2\pi]$ for $\lambda=\lambda'_j$ and that
  $Z(x+\tau)$ is an eigenfunction to the eigenvalue $\mu_{2k-1}(\tau)$.
  Thus showing that $\mu_{2k-1}(\tau)$ attains both $\lambda'_{2k-1}$ and $\lambda'_{2k}$.  
 \qed
\end{proof}

{\bf Remark }
In the above theorem we have that $\mu_0(\tau)\in [\lambda_{-1},\lambda_0]$, but 
 in general $\lambda_{-1}$ is not the minimum of $\mu_0(\tau)$ nor is $\lambda_0$
 the maximum of $\mu_0(\tau)$.  To see this consider the following example.

{\bf Example} Consider the case of $Q(t)=\left[\begin{array}{cc} 0 & 1 \\ 1 & 0 \end{array}\right]$ then
 $\mu_0(\tau)=0=\nu_0(\tau)$ for all $\tau$, but $\Delta(0)=2\cosh(\pi)>2$ so $\lambda=0$ is not an eigenvalue of the
 periodic problem.  Thus here we have
        \begin{eqnarray*}
                \lambda_{-1} < \inf_{\tau}\mu_{0}(\tau)=0= \sup_{\tau}\mu_{0}(\tau)<\lambda_0.
        \end{eqnarray*}

{\bf Remark }
 If $Q(x)$ is constant then $\mu_n(\tau)$ and $\nu_n(\tau)$ are independent of $\tau$ and from the above
 all the instability intervals vanish except possibly $I_0=(\lambda_{-1},\lambda_0)$.
 Our main result, in Section 6, gives a partial converse to this.

%%%%%%%%%%%%%%%%%%% asymptotics %%%%%%%%%%%%%%%%%%%%%%
\newsection{Solution asymptotics\label{sec-SolnAsymp}}

We say that the potential $Q$ is in canonical form if
\begin{eqnarray}
  Q(z) = \left( \begin{array}{cc} q_1(z)  & q_2(z)  \\
  q_2(z)  &  -q_1(z)  \end{array} \right), \label{canonical}
\end{eqnarray} 
where $p_1$ and $p_2$ are real valued measurable functions. A direct computation gives that if
$Q$ is in canonical form then $JQ=-QJ$.
Through out the remainder the norm of a matrix denotes the
operator matrix norm $$|[c_{ij}]|=\max_{j}\sqrt{\sum_i |c_{ij}|^2}.$$
Solution asymptotics will be given only for the case of (\ref{MainDifferEqn})  
with potential in canonical form as these are all that are required for the study of the inverse 
problem.

\begin{thm}\label{CanonicalAsympts}
  Let $Q$ be in canonical form with entries
  absolutely continuous and $|Q'|$ integrable on $[0,\pi]$. 
  The matrix solution $\mathbb{Y}$ of (\ref{MainDifferEqn}) with initial condition
   (\ref{initialConds})
  is of order $1$ and for $|\lambda|$ large takes the asymptotic form
\begin{eqnarray*}
  \mathbb{Y}(z) =  e^{-\lambda Jz}\left(I-\frac{Q(0)}{2\lambda}\right)
     + \frac{Q(z)e^{-\lambda Jz}}{2\lambda} 
     + \int_0^z \frac{e^{\lambda J(t-z)}}{2\lambda}(JQ^2-Q')e^{-\lambda Jt}\,dt
     +O\left(\frac{e^{|\Im\lambda|z}}{\lambda^2}\right).
\end{eqnarray*}
\end{thm}

\begin{proof}
 Using variation of parameters we can represent equation (\ref{MainDifferEqn})
 as the integral equation
\begin{equation}
 \mathbb{Y}(z) = e^{-\lambda J z}
    + \int_0^z e^{-\lambda J(z-t)}JQ\mathbb{Y} dt. \label{var-par-1}
\end{equation}
In the above equation take 
$\mathbb{Y}(z)=e^{|\Im\lambda|z}\mathbb{V}(z)$ giving
\begin{equation}
 \mathbb{V}(z) = e^{-|\Im \lambda| z}e^{-\lambda J z}
    + \int_0^z e^{-|\Im \lambda| (z-t)}e^{-\lambda J(z-t)}JQ\mathbb{V}(t) dt. \label{var-par-2}
\end{equation}
From (\ref{var-par-2}) we have
\begin{equation}
 |\mathbb{V}(z)| \le 1
    + \int_0^z |Q||\mathbb{V}(t)| dt. \label{var-par-3-i}
\end{equation}
Applying Gronwall's Lemma \cite[Lemma 6.3.6]{hormander} to (\ref{var-par-3-i})
gives
\begin{eqnarray*}
 |\mathbb{V}(z)| 
  \le \exp\left(\int_0^\pi |Q|\, dt\right).
\end{eqnarray*}
Hence $\mathbb{V}(z) = O(1)$ and thus $\mathbb{Y}(z) = O(e^{|\Im\lambda|z})$.

Let $\mathbb{Y}(z) = e^{-\lambda Jz}\mathbb{W}(z)$. Direct computation shows that 
$J Q = -Q J$. 
In terms of $\mathbb{W}$ (\ref{MainDifferEqn}) becomes
\begin{equation}
 \mathbb{W}' = Je^{2\lambda J z}Q\mathbb{W}. \label{var-par-3}
\end{equation}
and thus
\begin{equation}
 Q\mathbb{W}' = -Je^{-2\lambda J z}Q^2\mathbb{W}. \label{var-par-4}
\end{equation}
Integrating (\ref{var-par-3})  from $0$ to $z$ gives
\begin{equation}
 \mathbb{W}(z) = \mathbb{W}(0)
     + \frac{1}{2\lambda}\int_0^z \frac{d(e^{2\lambda J t})}{dt}Q(t)\mathbb{W}(t)\,dt, 
  \label{var-par-5}
\end{equation}
which, when integrated by parts, yields
\begin{equation}
 \mathbb{W}(z) = \mathbb{W}(0)
     + \frac{1}{2\lambda}\left[e^{2\lambda J t}Q\mathbb{W}\right]_0^z 
     - \frac{1}{2\lambda}\int_0^z e^{2\lambda J t}(Q'\mathbb{W}+Q\mathbb{W}')\,dt. 
  \label{var-par-6}
\end{equation}
Combining (\ref{var-par-4}) and (\ref{var-par-6}) gives
\begin{equation}
 \mathbb{W}(z) = \mathbb{W}(0)
     + \frac{1}{2\lambda}\left[e^{2\lambda J t}Q\mathbb{W}\right]_0^z 
     - \frac{1}{2\lambda}\int_0^z e^{2\lambda J t}Q'\mathbb{W}\,dt 
     + \frac{J}{2\lambda}\int_0^z Q^2\mathbb{W}\,dt. 
  \label{var-par-7}
\end{equation}
Thus
\begin{equation}
  \mathbb{Y}(z) =  e^{-\lambda Jz}\left(I-\frac{Q(0)}{2\lambda}\right)
     + \frac{Q(z)\mathbb{Y}(z)}{2\lambda} 
     + \int_0^z \frac{e^{\lambda J(t-z)}}{2\lambda}(JQ^2-Q')\mathbb{Y}\,dt.
\label{var-par-9}
\end{equation}
Here $|\mathbb{Y}(t)|=O(e^{|\Im\lambda|t})$ and
$|e^{\lambda J(t-z)}|=O(e^{|\Im\lambda|(z-t)})$ for $0\le t\le z$, so from (\ref{var-par-9})
\begin{equation}
  \mathbb{Y}(z) =  e^{-\lambda Jz}+O\left(\frac{e^{|\Im\lambda|z}}{\lambda}\right).
\label{var-par-10}
\end{equation}
Substituting (\ref{var-par-10}) into (\ref{var-par-9}) gives
\begin{eqnarray*}
  \mathbb{Y}(z) =  e^{-\lambda Jz}\left(I-\frac{Q(0)}{2\lambda}\right)
     + \frac{Q(z)e^{-\lambda Jz}}{2\lambda} 
     + \int_0^z \frac{e^{\lambda J(t-z)}}{2\lambda}(JQ^2-Q')e^{-\lambda Jt}\,dt
     +O\left(\frac{e^{|\Im\lambda|z}}{\lambda^2}\right)
\end{eqnarray*}
proving the theorem.
\qed
\end{proof}

Applying the Riemann-Lebesgue Lemma \cite{mcshane} to  
Theorem \ref{CanonicalAsympts} gives the courser but simpler asymptotic approximation
\begin{eqnarray}
  \mathbb{Y}(z) =  e^{-\lambda Jz}\left[I-\frac{Q(0)}{2\lambda}
    + \frac{J}{2\lambda}\int_0^z Q^2\,dt
 \right]
     + e^{\lambda Jz}\frac{Q(z)}{2\lambda} 
     +o\left(\frac{e^{|\Im\lambda|z}}{\lambda}\right).\label{course-asymp}
\end{eqnarray}

%%%%%%%%%%%%%% main results %%%%%%%%%%%%%%%%%%%%%%%%%%
\newsection{Inverse problem \label{sec-MainRes}}
We are now in a position to characterize the class of real symmetric matrices, $Q$, with 
absolutely continuous entries for which the instability intervals of (\ref{MainDifferEqn}) vanish,
see Theorem~\ref{main-thm}.

\begin{lem}\label{thrm1.1}
  Let $Q$ be in canonical form and have absolutely continuous entries which are 
 $\pi$-periodic on $\R$. 
 All instability intervals of (\ref{MainDifferEqn}) vanish if and only if
 $Q = 0$.
\end{lem}

\proof 
If $Q(z) = {0}$, 
then $\mathbb{Y}(z) = e^{-\lambda J z}$ 
 giving $\Delta(\lambda) = 2\cos(\lambda z)$. 
 So for all real $\lambda$, $|\Delta|\leq 2$. Thus all instability intervals vanishes. 

From the converse, suppose that all instability intervals of (\ref{MainDifferEqn}) vanish.
Now  $\lambda'_{2k-1}=\lambda'_{2k}$ and
 $\lambda_{2k-1}=\lambda_{2k}$, $k\in\Z$. 
So from 
  (\ref{indexing})
  $\lambda'_{2k-1}=\mu_{2k-1}(\tau)=\nu_{2k-1}(\tau)=\lambda'_{2k}$ and
 $\lambda_{2k-1}=\mu_{2k}(\tau)=\nu_{2k}(\tau)=\lambda_{2k}$ for all $\tau\in\R$. 
In the notation of Lemma \ref{invarDiscrim}, as a consequence of the above equality,
the $\lambda$-zeros of the entire functions $u_{ij}(\pi,\tau), i\ne j$, are
 $\{\lambda_{2k} |k\in\Z\}\cup\{\lambda'_{2k} |k\in\Z\}$,
for each $\tau\in\R$.
In addition the zeros of $u_{ij}(\pi,\tau), i\ne j,$ are simple for each $\tau\in\R$.  
Here $[u_{ij}]_{(j,i)}=\mathbb{U}$.
Thus $u_{ij}(\pi,\tau)/u_{ij}(\pi,0)$, for each $\tau\in\R$ and $i\ne j$, is an entire function of
$\lambda$.
However, from Theorem \ref{CanonicalAsympts}, 
\begin{eqnarray}
 {u}_{ij}(\pi,\tau)=(-1)^j\sin\lambda\pi+O\left(\frac{e^{|\Im\lambda|\pi}}{\lambda}\right),\quad 
 i\ne j.
 \label{asymp-order-1}
\end{eqnarray}
 Let $\Gamma_n, n\in\N,$ denote the closed paths in $\C$ consisting of
 the squares with corners at $2n(1\pm {\rm i})+\frac{1}{2}$ and 
  $-2n(1\mp{\rm i})+\frac{1}{2}$.
 On the edges of $\Gamma_n$ parametrized by $\lambda=\pm (2n +{\rm i}t)+\frac{1}{2}, t\in [-2n,2n],$ 
for $i\ne j$, we have
\begin{eqnarray*}
   {u}_{ij}(\pi,\tau)
   =(-1)^j\cosh \pi t+O\left(\frac{e^{\pi |t|}}{n}\right)
   =(-1)^j\frac{e^{\pi |t|}}{2}\left(1+O\left(\frac{1}{n}\right)\right),
\end{eqnarray*}
 giving
\begin{eqnarray*}
  \frac{ u_{ij}(\pi,\tau)}{u_{ij}(\pi,0)}
   =1+O\left(\frac{1}{n}\right).
\end{eqnarray*}
 On the edges of $\Gamma_n$ parametrized by $\lambda=\pm(2n{\rm i}-t)+\frac{1}{2},
 t\in [-2n,2n],$ we have
\begin{eqnarray*}
   {u}_{ij}(\pi,\tau)
   =(-1)^j\cos \pi (t-2ni)+O\left(\frac{e^{\pi |t|}}{n}\right)
   =(-1)^j\frac{e^{2\pi n}}{2}\left(e^{\pi{\rm i} t}+O\left(\frac{1}{n}\right)\right),
\end{eqnarray*}
giving
\begin{eqnarray*}
  \frac{ {u}_{ij}(\pi,\tau)}{{u}_{ij}(\pi,0)}
   =1+O\left(\frac{1}{n}\right).
\end{eqnarray*}
Thus by the maximum modulus principal, for $i\ne j$,
\begin{eqnarray*}
  \left|\frac{ {u}_{ij}(\pi,\tau)}{{u}_{ij}(\pi,0)}-1\right|
   =O\left(\frac{1}{n}\right).
\end{eqnarray*}
 on the region enclosed by $\Gamma_n$ for each $n\in\N$.
Taking $n\to\infty$ gives
\begin{eqnarray*}
  \frac{ {u}_{ij}(\pi,\tau)}{{u}_{ij}(\pi,0)}=1,\quad i\ne j,
\end{eqnarray*}
 on $\C$, and ${u}_{ij}(\pi,\tau)={u}_{ij}(\pi,0)$, for all $\tau\in\R$, $i\ne j$,
 on $\C$.
By Lemma \ref{invarDiscrim}, 
 $\Delta(\lambda,\tau)= \Delta(\lambda)$ for $\tau\in\R$ and $\lambda\in\C$.
Thus, as functions of $\lambda$, we have 
\begin{eqnarray}
{u}_{ij}(\pi,\tau) &=& {y}_{ij}(\pi)  \quad \mbox{ for } i \neq j, \\
{u}_{11}(\pi,\tau) + {u}_{22}(\pi,\tau) &=& {y}_{11}(\pi) + {y}_{22}(\pi).
\end{eqnarray}
Setting $\gamma(\tau,\lambda):={u}_{11}(\pi,\tau)- {y}_{11}(\pi)$
it follows that
${u}_{22}(\pi,\tau) = {y}_{22}(\pi)-\gamma(\tau,\lambda)$ and
\begin{eqnarray}
\mathbb{U}(\pi, \tau) = \mathbb{Y}(\pi) +\gamma(\tau,\lambda)\sigma_3,\label{1-6-2013-1}
\end{eqnarray} 
where $\sigma_3 = \left( \begin{array}{cc}1 &  0\\0 & -1 \end{array} \right)$.
Combining (\ref{course-asymp}) and (\ref{1-6-2013-1}) gives
\begin{eqnarray}
 \gamma(\tau,\lambda)\sigma_3=
 \mathbb{U}(\pi,\tau)- \mathbb{Y}(\pi) =  \frac{\sin\lambda\pi}{\lambda} J
     \left[Q(\tau) -Q(0) \right]
      +o\left(\frac{e^{|\Im\lambda|\pi}}{\lambda}\right).\label{gamma-asymp}
\end{eqnarray}

Equating the off diagonal components in (\ref{gamma-asymp}) yields in the notation of (\ref{canonical})
\begin{eqnarray}\label{finalstep}
(q_1(\tau)-q_1(0))\sin(\lambda\pi)=o\left(e^{|\Im\lambda|\pi}\right).
\end{eqnarray}
Now setting $\lambda=2n+\frac{1}{2}$ for $n\in\N$ in (\ref{finalstep})
gives  $q_1(\tau)-q_1(0)=o\left(1\right)$, from which it follows that $q_1(\tau)=q_1(0)$
for all $\tau\in\R$.  Hence $q_1$ is constant.

Let $\tilde{Y}(z) = e^{J\omega}Y(z)$, $\omega \in\R$.
This unitary transformation transforms (\ref{MainDifferEqn}) to 
\begin{equation}\label{MainThrmEval2}
	J\tilde{Y}' + \tilde{Q}\tilde{Y} = \lambda \tilde{Y},
\end{equation}
where $\tilde{Q} = e^{2J\omega}Q$.
Such unitary transformations are isospectral, thus the periodic
 eigenvalues of problem (\ref{MainDifferEqn}) and (\ref{MainThrmEval2}) are the same,
and similarly for the antiperiodic eigenvalues, see \cite[Ch. 7.1]{LaS}.
Setting $\omega = \pi/4$, then, in the notation of (\ref{canonical}),
\begin{eqnarray}\label{secondFinalStep}
\tilde{Q}(z) = \left( \begin{array}{cc} q_2(z)  & -q_1(z)  \\
	  -q_1(z)  &  -q_2(z)  \end{array} \right),
\end{eqnarray}
which is in canonical form.
The first part of the proof can now be applied to (\ref{MainThrmEval2}) to give $q_2$ constant.

Having established that $q_1$ and $q_2$, and thus $Q$, are constant 
we set $\omega = \frac{1}{2}\arctan\left(\frac{q_2}{q_1}\right)$ in the above transformation, to
give
\begin{equation}
 \tilde{Q} = m\sigma_3
	\quad\mbox{where}\quad m = \sqrt{q_1^2 + q_2^2}.\label{free}
\end{equation}  
Equation (\ref{MainThrmEval2}) with $\tilde{Q}$ as in (\ref{free}),
is the free particle Dirac system studied in \cite[Appendix]{YY}.
Using the fundamental matrix obtained in \cite[Appendix]{YY}, or by direct computation, we
 have that $$\Delta(\lambda) = 2\cos\sqrt{(\lambda^2 - m^2)}\pi.$$ 
Since, by assumption, all instability intervals vanish $|\Delta(\lambda)|\leq2$ for all $\lambda\in\R$. 
In particular $|\Delta(0)|\leq 2$, giving $\cosh(m \pi ) \le 1$ and so $m=0$. Thus $Q = { 0}$. 
\qed

\begin{lem}\label{thrm1.2}
 Let $Q$ be a symmetric matrix with real valued absolutely continuous $\pi$-periodic entries.
 If all instability intervals of (\ref{MainDifferEqn}) vanish then $Q=pI$ where
 $p:=\frac{{\rm trace}(Q)}{2}$.  In this case
$\lambda_{2k-1}=\lambda_{2k}=2k+\frac{1}{\pi}\int_0^\pi p\,dt=  1+\lambda_{2k-1}'=1+\lambda_{2k}',$ for  $k\in\Z.$ 
\end{lem}

\proof 
 Let $$h(z)=\frac{\pi-z}{\pi}\int_0^z p\,dt
                  -\frac{z}{\pi}\int_z^\pi p\,dt,$$
then $h(0)=0$ and $h(\pi)=0$, so $h$ can be extended to a $\pi$-periodic function on $\R$.
 Let $Y(z)=e^{Jh(z)}X(z)$ then $Y(0)=X(0)$ and $Y(\pi)=X(\pi)$, so the transformation
 preserves boundary conditions.  Here $X(z)$ obeys the equation
 \begin{eqnarray}\label{tx-eq}
   JX'+\tilde{Q}X=\tilde{\lambda} X
 \end{eqnarray}
  where
 \begin{eqnarray}\label{tx-coeff}
   \tilde{Q}&=&e^{-Jh(z)}\left(Q(z)-pI\right)e^{Jh(z)},\\
    \tilde{\lambda}&=&\lambda-\frac{1}{\pi}\int_0^\pi p\,dt,\label{ep}
  \end{eqnarray}
 and $\tilde{Q}$ is a real symmetric matrix valued function with $\pi$-periodic absolutely
 continuous entries. 
 As ${\rm trace}\left(Q(z)-pI\right)=0$ we have
 ${\rm trace}(\tilde{Q})=0$ and $\tilde{Q}$ is in canonical form. 
 In addition the $\tilde{\lambda}$-eigenvalues of (\ref{tx-eq}) 
with periodic and anti-periodic boundary
 conditions are precisely the $\lambda$-eigenvalues of (\ref{MainDifferEqn}) with respectively
 periodic and anti-periodic boundary conditions, but shifted by 
 $-\frac{1}{\pi}\int_0^\pi p\,dt$.  
 If all instability intervals of 
 (\ref{MainDifferEqn}) vanish, so do those of (\ref{tx-eq}).
 Lemma \ref{thrm1.1} can now be applied to (\ref{tx-eq}) to give $\tilde{Q}=0$.
 Hence $Q(z)=pI$, from which the first claim of the lemma follows.
 In this case direct computation gives
 \begin{eqnarray}
   \tilde{\Delta}(\tilde{\lambda}) = 2\cos\tilde{\lambda}\pi,\label{delta-tilde}
 \end{eqnarray}
 where $\tilde{\Delta}$ is the descriminant of (\ref{tx-eq}).
 From Section 4, (\ref{delta-tilde}) and direct computation we see that for $\tilde{Q}=0$, 
$\tilde{\lambda}_{2k-1}=\tilde{\lambda}_{2k}=2k$ and
  by (\ref{indexing}) and (\ref{delta-tilde}), 
$\tilde{\lambda}_{2k-1}'=\tilde{\lambda}_{2k}'=2k-1$, $k\in\Z$, from which
 along with (\ref{ep}) the remaining claims of the lemma follow.
\qed

\begin{lem}\label{lem-converse}
 If $p$ is a real (scalar) valued $\pi$-periodic function which is integrable on compact sets then
 all instability intervals vanish for the equation 
 \begin{eqnarray}
    JY'+pY=\lambda Y.\label{eq-converse}
 \end{eqnarray}
\end{lem}

\proof
  A direct computation yields that for (\ref{eq-converse}) we have 
  $$\mathbb{Y}(z)=e^{J(\int_0^z p\,dt - \lambda z)}.$$
  Taking the trace of $\mathbb{Y}(\pi)$ gives
  $$\Delta(\lambda)=2\cos\left(\lambda\pi-\int_0^\pi p\,dt\right)$$
  from which it follows that all instability intervals vanish.
\qed

Combining Lemma \ref{thrm1.2} and Lemma \ref{lem-converse} we obtain our main theorem.

\begin{thm}\label{main-thm}
 Let $Q$ be a real symmetric matrix valued function with absolutely continuous $\pi$-periodic
 entries.  All instability intervals of (\ref{MainDifferEqn}) vanish if and only if
 $Q=pI$ for some absolutely continuous real (scalar) valued $\pi$-periodic function $p$.
\end{thm}

%%%%%%%%%%%%%%%%% bibliography %%%%%%%%%%%%%%%%%%%%%%

%%%%%%%%%%%%%%%%%%%%%%%%
\end{document}